\documentclass[11pt]{amsart}
\usepackage{mathrsfs,latexsym,amsfonts,amssymb}
\usepackage{hyperref}
\newtheorem{theorem}{Theorem}[section]
\newtheorem{lemma}[theorem]{Lemma}
\newtheorem{corollary}[theorem]{Corollary}
\newtheorem{question}[theorem]{Question}
\newtheorem{example}[theorem]{Example}
\theoremstyle{definition}
\newtheorem{definition}[theorem]{Definition}
\newtheorem{proposition}[theorem]{Proposition}
\theoremstyle{remark}

\begin{document}

\title[Submaximal properties in (strongly) topological gyrogroups]
{Submaximal properties in (strongly) topological gyrogroups}

\author{Meng Bao}
\address{(Meng Bao): 1. School of mathematics and statistics,
Minnan Normal University, Zhangzhou 363000, P. R. China; 2. College of Mathematics, Sichuan University, Chengdu 610064, P. R. China}
\email{mengbao95213@163.com}

\author{Fucai Lin*}
\address{(Fucai Lin): School of mathematics and statistics,
Minnan Normal University, Zhangzhou 363000, P. R. China}
\email{linfucai2008@aliyun.com; linfucai@mnnu.edu.cn}

\thanks{The authors are supported by the Key Program of the Natural Science Foundation of Fujian Province (No: 2020J02043), the NSFC (No. 11571158), the lab of Granular Computing, the Institute of Meteorological Big Data-Digital Fujian and Fujian Key Laboratory of Data Science and Statistics.\\
*corresponding author}


\begin{abstract}
A space $X$ is submaximal if any dense subset of $X$ is open. In this paper, we prove that every submaximal topological gyrogroup of non-measurable cardinality is strongly $\sigma$-discrete. Moreover, we prove that every submaximal strongly topological gyrogroup of non-measurable cardinality is hereditarily paracompact.
\end{abstract}

\maketitle
\section{Introduction}

In the 1940s, E. Hewitt in \cite{He} introduced the concepts of maximality and submaximality of general topological spaces, which are important tools to deal with extreme cases when studying the family of topologies without isolated points. At the same time, E. Hewitt found a general way to construct maximal and submaximal topologies. He constructed maximal and submaximal topologies by transfinite induction and by the method that any chain of topologies on the same set has an upper bound with the same separation axioms. Then, A.V. Arhangel' ski\v\i ~and P.J. Collins \cite{AP} began to study the class of submaximal spaces systematically in 1995 and give some necessary and sufficient conditions for a space to be submaximal. In 1998, O. Alas, I. Protasov, M. Tkachenko, etc.\cite{APT} studied the maximal and submaximal groups and proved that every submaximal topological group of non-measurable cardinality is strongly $\sigma$-discrete, and every submaximal strongly topological group of non-measurable cardinality is hereditarily paracompact.

The gyrogroup was first introduced by A.A. Ungar in \cite{UA1988} when he researched $c$-ball of relativistically admissible velocities with Einstein velocity addition. The Einstein velocity addition $\oplus _{E}$ is given as follows: $$\mathbf{u}\oplus _{E}\mathbf{v}=\frac{1}{1+\frac{\mathbf{u}\cdot \mathbf{v}}{c^{2}}}(\mathbf{u}+\frac{1}{\gamma _{\mathbf{u}}}\mathbf{v}+\frac{1}{c^{2}}\frac{\gamma _{\mathbf{u}}}{1+\gamma _{\mathbf{u}}}(\mathbf{u}\cdot \mathbf{v})\mathbf{u}),$$ where $\mathbf{u,v}\in \mathbb{R}_{c}^{3}=\{\mathbf{v}\in \mathbb{R}^{3}:||\mathbf{v}||<c\}$ and $\gamma _{\mathbf{u}}$ is given by $$\gamma _{\mathbf{u}}=\frac{1}{\sqrt{1-\frac{\mathbf{u}\cdot \mathbf{u}}{c^{2}}}}.$$ From \cite{UA2002}, we know that the gyrogroup has a weaker algebraic structure than a group. In 2017, W. Atiponrat \cite{AW} gave the concept of topological gyrogroups. Then Z. Cai, S. Lin and W. He in \cite{CZ} proved that every topological gyrogroup is a rectifiable space. In 2019, the authors \cite{BL} defined the concept of strongly topological gyrogroups, and proved that every feathered strongly topological gyrogroup is paracompact. Moreover, the authors proved that every strongly topological gyrogroup with a countable pseudocharacter is submetrizable and every locally paracompact strongly topological gyrogroup is paracompact, see \cite{BL1,BL2}.

In this paper, we generalize some well known results in the class of submaximal topological groups to submaximal topological gyrogroups. In particular, we prove that every submaximal topological gyrogroup of non-measurable cardinality is strongly $\sigma$-discrete and every submaximal strongly topological gyrogroup of non-measurable cardinality is hereditarily paracompact, which generalizes some results in \cite{APT}.

\smallskip
\section{Preliminaries}
Throughout this paper, all topological spaces are assumed to be Hausdorff and dense in themselves, unless otherwise is explicitly stated. Let $\mathbb{N}$ be the set of all positive integers and $\omega$ the first infinite ordinal. Let $X$ be a topological space and $A \subseteq X$ be a subset of $X$.
  The {\it closure} of $A$ in $X$ is denoted by $\overline{A}$ and the
  {\it interior} of $A$ in $X$ is denoted by $\mbox{Int}(A)$. A cardinal number $m$ is called {\it non-measurable} \cite{E} provided that the only countably additive two-valued measure defined on the family of all subsets of a set $X$ of cardinality $m$ which vanishes on all one-point sets is the trivial measure, identically equal to zero.
  The readers may consult \cite{AA, E, linbook} for notation and terminology not explicitly given here.

\begin{definition}\cite{AW}
Let $G$ be a nonempty set, and let $\oplus: G\times G\rightarrow G$ be a binary operation on $G$. Then the pair $(G, \oplus)$ is called a {\it groupoid}. A function $f$ from a groupoid $(G_{1}, \oplus_{1})$ to a groupoid $(G_{2}, \oplus_{2})$ is called a {\it groupoid homomorphism} if $f(x\oplus_{1}y)=f(x)\oplus_{2} f(y)$ for all elements $x, y\in G_{1}$. Furthermore, a bijective groupoid homomorphism from a groupoid $(G, \oplus)$ to itself will be called a {\it groupoid automorphism}. We write $\mbox{Aut}(G, \oplus)$ for the set of all automorphisms of a groupoid $(G, \oplus)$.
\end{definition}

\begin{definition}\cite{UA}
Let $(G, \oplus)$ be a groupoid. The system $(G,\oplus)$ is called a {\it gyrogroup}, if its binary operation satisfies the following conditions:

\smallskip
(G1) There exists a unique identity element $0\in G$ such that $0\oplus a=a=a\oplus0$ for all $a\in G$.

\smallskip
(G2) For each $x\in G$, there exists a unique inverse element $\ominus x\in G$ such that $\ominus x \oplus x=0=x\oplus (\ominus x)$.

\smallskip
(G3) For all $x, y\in G$, there exists $\mbox{gyr}[x, y]\in \mbox{Aut}(G, \oplus)$ with the property that $x\oplus (y\oplus z)=(x\oplus y)\oplus \mbox{gyr}[x, y](z)$ for all $z\in G$.

\smallskip
(G4) For any $x, y\in G$, $\mbox{gyr}[x\oplus y, y]=\mbox{gyr}[x, y]$.
\end{definition}

\begin{lemma}\cite{UA}\label{a}
Let $(G, \oplus)$ be a gyrogroup. Then for any $x, y, z\in G$, we obtain the following:

\begin{enumerate}
\smallskip
\item $(\ominus x)\oplus (x\oplus y)=y$. \ \ \ (left cancellation law)

\smallskip
\item $(x\oplus (\ominus y))\oplus gyr[x, \ominus y](y)=x$. \ \ \ (right cancellation law)

\smallskip
\item $(x\oplus gyr[x, y](\ominus y))\oplus y=x$.

\smallskip
\item $gyr[x, y](z)=\ominus (x\oplus y)\oplus (x\oplus (y\oplus z))$.
\end{enumerate}
\end{lemma}

The definition of a subgyrogroup is given as follows.

\begin{definition}\cite{ST}
Let $(G,\oplus)$ be a gyrogroup. A nonempty subset $H$ of $G$ is called a {\it subgyrogroup}, denoted
by $H\leq G$, if $H$ forms a gyrogroup under the operation inherited from $G$ and the restriction of $gyr[a,b]$ to $H$ is an automorphism of $H$ for all $a,b\in H$.

\smallskip
Furthermore, a subgyrogroup $H$ of $G$ is said to be an {\it $L$-subgyrogroup}, denoted
by $H\leq_{L} G$, if $gyr[a, h](H)=H$ for all $a\in G$ and $h\in H$.
\end{definition}

The subgyrogroup criterion is given in \cite{ST} (that is, a nonempty subset $H$ of a gyrogroup $G$ is a subgyrogroup if and only if $\ominus a\in H$ and $a\oplus b\in H$ for all $a, b\in H$) which explains that by the item (4) in Lemma \ref{a} it follows the subgyrogroup criterion.

\begin{definition}\cite{AW}
A triple $(G, \tau, \oplus)$ is called a {\it topological gyrogroup} if the following statements hold:

\smallskip
(1) $(G, \tau)$ is a topological space.

\smallskip
(2) $(G, \oplus)$ is a gyrogroup.

\smallskip
(3) The binary operation $\oplus: G\times G\rightarrow G$ is jointly continuous while $G\times G$ is endowed with the product topology, and the operation of taking the inverse $\ominus (\cdot): G\rightarrow G$, i.e. $x\rightarrow \ominus x$, is also continuous.
\end{definition}

It is clear that each topological group is a topological gyrogroup. However, every topological gyrogroup whose gyrations are not identically equal to the identity is not a topological group.

\begin{example}\cite{AW}
The Einstein gyrogroup with the standard topology is a topological gyrogroup but not a topological group.
\end{example}

The Einstein gyrogroup has been introduced in the Introduction. It was proved in \cite{UA} that $(\mathbb{R}^{3}_{c},\oplus _{E})$ is a gyrogroup but not a group. Moreover, with the standard topology inherited from $\mathbb{R}^{3}$, it is clear that $\oplus _{E}$ is continuous. Finally, $-\mathbf{u}$ is the inverse of $\mathbf{u}\in \mathbb{R}^{3}$ and the operation of taking the inverse is also continuous. Therefore, the Einstein gyrogroup $(\mathbb{R}^{3}_{c},\oplus _{E})$ with the standard topology inherited from $\mathbb{R}^{3}$ is a topological gyrogroup but not a topological group.

\begin{definition}\cite{He}
A topological space $(X,\tau)$ is called {\it maximal} if for any topology $\mu$ on $X$ strictly finer that $\tau$, the space $(X,\mu)$ has an isolated point. A space $X$ is {\it submaximal} if any dense subset of $X$ is open.
\end{definition}

\begin{definition}\cite{APT}
A non-empty family $\mathcal{D}$ of dense subsets of a space $X$ is called a {\it filter of dense subsets} of $X$ if $\mathcal{D}$ is closed with respect to finite intersections and $D\in \mathcal{D}$, $D\subset D_{1}\subset X$ implies $D_{1}\in \mathcal{D}$. The family $\mathcal{D}$ is called an {\it ultrafilter of dense subsets} of $X$ if there is no filter of dense subsets of $X$ that properly contains $\mathcal{D}$.
\end{definition}

\begin{definition}\cite{APT}
A space is called {\it irresolvable} if it is not the union of two disjoint dense subsets.
\end{definition}

\begin{definition}\cite{E}
A space $X$ is called {\it collectionwise Hausdorff} if for any discrete subset $A$ of $X$ it is possible to choose an open set $V_{p}$ containing $p$ for every $p\in A$ in such a way that the family $\{V_{p}:p\in A\}$ is discrete.
\end{definition}

\begin{definition}\cite{E}\label{d00}
If $X$ is a space and $x\in X$, then the {\it dispersion character} $\Delta (x,X)$ of $X$ at the point $x$ is the minimum of the cardinalities of open subsets of $X$ containing $x$. The cardinal number $\Delta (X)=min\{\Delta (x,X):x\in X\}$ is called the {\it dispersion character} of $X$.
\end{definition}

Next, we recall the definition of strongly topological gyrogroups.

\begin{definition}\cite{BL}\label{d11}
Let $G$ be a topological gyrogroup. We say that $G$ is a {\it strongly topological gyrogroup} if there exists a neighborhood base $\mathscr U$ of $0$ such that, for every $U\in \mathscr U$, $\mbox{gyr}[x, y](U)=U$ for any $x, y\in G$. For convenience, we say that $G$ is a strongly topological gyrogroup with neighborhood base $\mathscr U$ of $0$.
\end{definition}

For each $U\in \mathscr U$, we can set $V=U\cup (\ominus U)$. Then, $$gyr[x,y](V)=gyr[x, y](U\cup (\ominus U))=gyr[x, y](U)\cup (\ominus gyr[x, y](U))=U\cup (\ominus U)=V,$$ for all $x, y\in G$. Obviously, the family $\{U\cup(\ominus U): U\in \mathscr U\}$ is also a neighborhood base of $0$. Therefore, we may assume that $U$ is symmetric for each $U\in\mathscr U$ in Definition~\ref{d11}.

In \cite{BL}, the authors proved that there is a strongly topological gyrogroup which is not a topological group, see Example \ref{lz1}.

\begin{example}\cite{BL}\label{lz1}
Let $\mathbb{D}$ be the complex open unit disk $\{z\in \mathbb{C}:|z|<1\}$. We consider $\mathbb{D}$ with the standard topology. In \cite[Example 2]{AW}, define a M\"{o}bius addition $\oplus _{M}: \mathbb{D}\times \mathbb{D}\rightarrow \mathbb{D}$ to be a function such that $$a\oplus _{M}b=\frac{a+b}{1+\bar{a}b}\ \mbox{for all}\ a, b\in \mathbb{D}.$$ Then $(\mathbb{D}, \oplus _{M})$ is a gyrogroup, and it follows from \cite[Example 2]{AW} that $$gyr[a, b](c)=\frac{1+a\bar{b}}{1+\bar{a}b}c\ \mbox{for any}\ a, b, c\in \mathbb{D}.$$ For any $n\in \mathbb{N}$, let $U_{n}=\{x\in \mathbb{D}: |x|\leq \frac{1}{n}\}$. Then, $\mathscr U=\{U_{n}: n\in \mathbb{N}\}$ is a neighborhood base of $0$. Moreover, we observe that $|\frac{1+a\bar{b}}{1+\bar{a}b}|=1$. Therefore, we obtain that $gyr[x, y](U)\subset U$, for any $x, y\in \mathbb{D}$ and each $U\in \mathscr U$, then it follows that $gyr[x, y](U)=U$ by \cite[Proposition 2.6]{ST}. Hence, $(\mathbb{D}, \oplus _{M})$ is a strongly topological gyrogroup. However, $(\mathbb{D}, \oplus _{M})$ is not a group \cite[Example 2]{AW}.
\end{example}

Indeed, it is well known that M\"{o}bius gyrogroups, Einstein gyrogroups, and Proper velocity gyrogroups, that were studied in \cite{FM, FM1,FM2,UA}, are all strongly topological gyrogroups. Therefore, they are all topological gyrogroups and rectifiable spaces. At the same time, all of them are not topological groups. Further, it was also proved in \cite[Example 3.2]{BL} that there exists a strongly topological gyrogroup which has an infinite $L$-subgyrogroup.

\smallskip
\section{submaximal properties of topological gyrogroups}
In this section, we mainly prove that every submaximal topological gyrogroup of non-measurable cardinality is strongly $\sigma$-discrete. First, we show that, for any cardinality $\kappa>\omega$, there exists a gyrogroup $G$ with subgyrogroup $H$ of the cardinality $\kappa$ such that $H$ is not a group.

\begin{example}
For any cardinality $\kappa>\omega$, there exists a gyrogroup $G$ with subgyrogroup $H$ of the cardinality $\kappa$ such that $H$ is not a group.
\end{example}

Let $\mathbb{D}$ be gyrogroup in Example~\ref{lz1} and let $\kappa$ be an infinite cardinal number. It follows from \cite[Theorem 2.1]{ST2} that $\mathbb{D}^{\kappa}$ is a gyrogroup. Fix a subset $X$ of the gyrogroup $\mathbb{D}^{\kappa}$ such that the cardinality of $X$ is equal to $\kappa$ and $X$ contains arbitrary three points $x=(x_{\alpha})_{\alpha<\kappa}$, $y=(y_{\alpha})_{\alpha<\kappa}$ and $z=(z_{\alpha})_{\alpha<\kappa}$ of $\mathbb{D}^{\kappa}$ such that there exists $\beta<\alpha$ with $x_{\beta}=1/2, y_{\beta}=i/2$ and $z_{\beta}=-1/2$. From the proof of \cite[Example 2]{AW}, we see that $x\oplus (y\oplus z)\neq (x\oplus y)\oplus z$. Put $H=\langle X\rangle$, that is, $H$ is a subgyrogroup generated from $X$. Then the cardinality of $H$ is also equal to $\kappa$. Moreover, since $x, y, z\in H$, it follows that $H$ is not a group.

\begin{proposition}\label{p0000}
Let $G$ be a gyrogroup of cardinality $\kappa>\omega$. Then, for any $\omega<\alpha<\kappa$, there exists a subgyrogroup $G_{\alpha}$ of $G$ with the cardinality $\alpha$.
\end{proposition}

\begin{proof}
Take an arbitrary subset $Y$ of $G$ such that $Y=\ominus Y$, $0\in Y$ and $|Y|=\alpha$. Let $Y_{0}=Y$. By induction, we assume that we have defined $Y_{1}, \ldots, Y_{n}$ of subsets of $G$ such that $Y_{i+1}=\ominus(Y_{i}\oplus Y_{i})\cup(Y_{i}\oplus Y_{i})$ for any $i=0, \ldots, n-1$. Let $Y_{n+1}=\ominus(Y_{n}\oplus Y_{n})\cup  (Y_{n}\oplus Y_{n})$. Clearly, the cardinality of each $Y_{n}$ is just $\alpha$. Put $G_{\alpha}=\bigcup_{n\in\mathbb{N}}Y_{n}$. Then $G_{\alpha}$ is a subgyrogroup of $G$ with the cardinality $\alpha$. Indeed, it is obvious that $|G_{\alpha}|=\alpha$. It suffices to prove that $G_{\alpha}$ is a subgyrogroup of $G$. By our construction of $G_{\alpha}$, we have $G_{\alpha}=\ominus G_{\alpha}$. Take any $x, y\in G_{\alpha}$. Then there exists $n\in \mathbb{N}$ such that $x, y\in Y_{n}$, hence $x\oplus y\in Y_{n+1}\subset G_{\alpha}$. Therefore, $G_{\alpha}$ is a subgyrogroup of $G$.
\end{proof}

Next we give some lemmas.

\begin{lemma}\label{yl1}
Let $G$ be a gyrogroup of cardinality $\kappa>\omega$. Then for each $\alpha <\kappa$ there are subsets $G_{\alpha}$ and $H_{\alpha}$ of $G$ with the following properties:

\smallskip
(1) $G_{\alpha}$ is a subgyrogroup of $G$ for all $\alpha<\kappa$;

\smallskip
(2) if $\alpha<\beta<\kappa$, we have $G_{\alpha}\subset G_{\beta}$ and $G_{\alpha}\not =G_{\beta}$;

\smallskip
(3) $|G_{\alpha}|=|\alpha|$ for all $\alpha<\kappa$;

\smallskip
(4) $G_{\alpha}=\bigcup \{H_{\upsilon}:\upsilon \leq \alpha\}$ for all $\alpha <\kappa$;

\smallskip
(5) $\bigcup \{H_{\alpha} :\alpha <\kappa\}=G$ and $H_{\alpha}\cap H_{\beta}=\emptyset$ if $\alpha \not =\beta$;

\smallskip
(6) if $g\in H_{\alpha}$ and $\alpha <\beta$, we have that $g\oplus H_{\beta}=H_{\beta}\oplus g=H_{\beta}$;

\smallskip
(7) $H_{\alpha}=\ominus H_{\alpha}$ for all $\alpha <\kappa$;

\smallskip
(8) if $A$ is a confinal subset of $\kappa$, the cardinality of $\bigcup \{H_{\alpha}:\alpha \in A\}$ is $\kappa$.

The family $\{H_{\alpha}:\alpha <\kappa\}$ is called a canonical decomposition of $G$.
\end{lemma}

\begin{proof}
Since $G$ is a gyrogroup of cardinality $\kappa$, let $G=\{g_{\alpha}:\alpha <\kappa\}$, where $g_{0}=0$ and $g_{\alpha}\not =g_{\beta}$ if $\alpha \not =\beta$. Let $G_{0}=\langle \{g_{0}\}\rangle$. Suppose that $\beta <\kappa$ and that for every $\alpha <\beta$ we have constructed a subgyrogroup $G_{\alpha}$ of $G$ which has the following properties:

\smallskip
(i) $G_{\alpha}\subset G_{\gamma}$ and $G_{\alpha}\not =G_{\gamma}$ if $\alpha <\gamma <\beta$;

\smallskip
(ii) $|G_{\alpha}|=|\alpha|$ for all $\alpha <\beta$;

\smallskip
(iii) $\{g_{\gamma}:\gamma <\alpha\}\subset G_{\alpha}$ for every $\alpha <\beta$.

\smallskip
Let $B_{\beta}=\bigcup \{G_{\alpha}:\alpha <\beta\}$. It follows from (ii) that $B_{\beta}\not =G$, hence there exists $$\beta ^{*}=\min\{\alpha <\kappa:g_{\alpha}\not \in B_{\beta}\}.$$ Set $G_{\beta}=\langle B_{\beta}\cup \{g_{\beta ^{*}}\}\rangle$. Therefore, by induction, we can obtain that the family $\{G_{\alpha}: \alpha<\kappa\}$ satisfies (i)-(iii) as well as the property $\bigcup\{G_{\alpha}: \alpha<\kappa\}$. For every $\alpha <\kappa$, let $H_{\alpha}=G_{\alpha}\setminus \bigcup \{G_{\beta}:\beta <\alpha\}$.

Obviously, the sets $G_{\alpha}$ and $H_{\alpha}$ satisfy (1)-(5) and (7). To see that (6) holds. Assume that $g\in H_{\alpha}$ and $\alpha<\beta$. Clearly, $g\in G_{\alpha}$ and $G_{\alpha}$ is a subgyrogroup of $G_{\gamma}$ for each $\alpha \leq \gamma \leq \beta$. Therefore, $g\oplus G_{\gamma}=G_{\gamma}\oplus g=G_{\gamma}$ for all $\gamma$, then
\begin{eqnarray}
g\oplus H_{\beta}&=&g\oplus (G_{\beta}\setminus\bigcup \{G_{\alpha}:\alpha <\beta\})\nonumber\\
&=&(g\oplus G_{\beta})\setminus\bigcup \{g\oplus G_{\alpha}:\alpha <\beta\}\nonumber\\
&=&G_{\beta}\setminus \bigcup \{G_{\alpha}:\alpha <\beta\}\nonumber\\
&=&H_{\beta}.\nonumber
\end{eqnarray}
Similar, $H_{\beta}\oplus g=H_{\beta}$, thus (6) holds. Finally, take any cofinal $A\subset \kappa$. It follows from $|H_{\alpha +1}|=|G_{\alpha}|=|\alpha|$ that $$|\bigcup \{H_{\alpha}:\alpha \in A\}|=|\bigcup \{G_{\alpha}:\alpha \in A\}|=|G|=\kappa.$$
\end{proof}

Let $G$ be a gyrogroup, and let $\tau$ be a topology on $G$. A {\it left topological gyrogroup} consists of a gyrogroup $G$ and a topology $\tau$ on the set $G$ such that for all $g\in G$, the left action $l_{g}: G\rightarrow G$, $x\mapsto g\oplus x$, is a continuous mapping of the space $G$ to itself. Similarly, we can define the concept of {\it right topological gyrogroups}. Clearly, each topological gyrogroup is not only a left topological gyrogroup but also a right topological gyrogroup.

Let $G$ be a gyrogroup of cardinality $\kappa>\omega$, and let $\{H_{\alpha}:\alpha <\kappa\}$ be a canonical decomposition of $G$. Then for each $A\subset\kappa$ put $H_{A}=\bigcup \{H_{\alpha}:\alpha \in A\}$.

\begin{lemma}\label{yl2}
Suppose that $(G,\tau,\oplus)$ is a non-discrete irresolvable left (or right) topological gyrogroup such that $|G|=\Delta (G,\tau,\oplus)=\kappa>\omega$, and suppose that $\{H_{\alpha}:\alpha <\kappa\}$ is a canonical decomposition of $G$. Then the following statements hold:

\smallskip
(1) For each subset $A\subset \kappa$ and any $g, h\in G$, the set $\left(h\oplus(g\oplus H_{A})\right)\setminus H_{A}$ has cardinality less than $\kappa$.

 \smallskip
(2) The family $\xi =\{A\subset \kappa :\mbox{Int}(H_{A})\not =\emptyset\}$ is a free ultrafilter on $\kappa$.
\end{lemma}

\begin{proof}
(1) Since $g, h\in G$, there exists $\alpha, \beta<\kappa$ such that $$g\in H_{\alpha}\subset G_{\alpha}=\bigcup \{H_{\upsilon}:\upsilon \leq \alpha\}$$ and $h\in H_{\beta}\subset G_{\beta}=\bigcup \{H_{\upsilon}:\upsilon \leq \beta\}$. If $\alpha <\gamma$ ($\beta<\gamma$), it follows from Lemma \ref{yl1} that $g\oplus H_{\gamma}=H_{\gamma}$ ($h\oplus H_{\gamma}=H_{\gamma}$). If $\alpha\geq\beta$, then $$g\oplus H_{A}\subset G_{\alpha}\cup\{H_{\nu}: \nu>\alpha, \nu\in A\}.$$ Hence $$h\oplus (g\oplus H_{A})\subset G_{\alpha}\cup\{H_{\nu}: \nu>\alpha, \nu\in A\},$$ then $\left(h\oplus (g\oplus H_{A})\right)\setminus H_{A}\subset G_{\alpha}$. If $\alpha<\beta$, then it also easily see that $$h\oplus (g\oplus H_{A})\subset G_{\beta}\cup\{H_{\nu}: \nu>\beta, \nu\in A\}.$$ Hence $$\left(h\oplus (g\oplus H_{A})\right)\setminus H_{A}\subset G_{\beta}.$$ Therefore, it follows from Lemma \ref{yl1} that$\left(h\oplus(g\oplus H_{A})\right)\setminus H_{A}$ has cardinality less than $\kappa$.

\smallskip
(2) It is clear that $H_{A}$ and $H_{\kappa \setminus A}$ are disjoint and $G=H_{A}\cup H_{\kappa \setminus A}$. Hence one of the sets $H_{A}$ or $H_{\kappa \setminus A}$ has non-empty interior as $G$ is irresolvable. Therefore, $A\in \xi$ or $\kappa \setminus A\in \xi$. Indeed, exactly one of the sets $A$ and $\kappa \setminus A$ belongs to $\xi$. Suppose not, then both $A$ and $\kappa \setminus A$ belong to $\xi$. In order to obtain a contradiction, it suffices to prove that $H_{A}$ and $H_{\kappa \setminus A}$ are dense in $G$. Indeed, we need only to consider the case of $H_{A}$.

Clearly, $U=\mbox{Int}(H_{A})\not =\emptyset$. If $U$ is not dense in $G$, there exists a non-empty open set $V\subset G$ such that $V\cap U=\emptyset$. For arbitrary $x\in U$ and $y\in V$, set $W=(y\oplus((\ominus x)\oplus U))\cap V$. Obviously, $W$ is open, non-empty and $|W|=\kappa$.

By (1), we see that $|W\setminus H_{A}|<\kappa$, hence $\mbox{Int}(W\setminus H_{A})=\emptyset$ in $G$ by our assumption. By the definitions of $W$ and $U$, it follows that $$W=(W\setminus H_{A})\cup (W\cap H_{A})\subset (W\setminus H_{A})\cup (H_{A}\setminus U),$$ where $W\setminus H_{A}$ and $H_{A}\setminus U$ are disjoint and both of them have empty interior and dense in $W$. Therefore, $W$ is resolvable. The gyrogroup $G$ can be covered by the all possible left translations of $W$, so $G$ is resolvable by \cite{CL} which proves that the union of resolvable spaces is also resolvable. This is a contradiction.

Therefore, if $U=\mbox{Int}(H_{A})\not =\emptyset$, then it follows that $U$ has to be dense in $G$. So $A$ or $\kappa \setminus A$ belong to $\xi$ and $\xi$ is an ultrafilter on $\kappa$. Further, $\xi$ is a free ultrafilter since each $H_{\alpha}$ does not belongs to $\xi$ by (5) of Lemma \ref{yl1}.
\end{proof}

\begin{theorem}\label{3dl1}
Suppose that $(G, \tau, \oplus)$ is a non-discrete irresolvable left (or right) topological gyrogroup of non-measurable cardinality such that $\Delta (G,\tau,\oplus)=\kappa$. If $\{H_{\alpha}:\alpha <\kappa\}$ is a canonical decomposition of $G$, then there is a family $\{A_{n}:n\in \mathbb{N}\}$ of subsets of $\kappa$ such that:

\smallskip
(1) $\bigcup \{A_{n}:n\in \mathbb{N}\}=\kappa$;

\smallskip
(2) every set $H_{A_{n}}=\bigcup \{H_{\alpha}:\alpha \in A_{n}\}$ is closed and nowhere dense in $G$;

\smallskip
(3) $\bigcup \{H_{A_{n}}:n\in \mathbb{N}\}=G$.

\smallskip
In particular, $(G,\tau,\oplus)$ is of first category.
\end{theorem}

\begin{proof}
It follows from Lemma \ref{yl2} that the family $\xi =\{A\subset \kappa: \mbox{Int}(H_{A})\not =\emptyset\}$ is a free ultrafilter on $\kappa$. Since $\kappa$ is a non-measurable cardinal, there is a family $\{B_{n}:n\in \mathbb{N}\}$ such that $B_{n}\in \xi$ for every $n$ and $\bigcap \{B_{n}:n\in \mathbb{N}\}=\emptyset$. Then $\bigcap \{H_{B_{n}}:n\in \mathbb{N}\}=\emptyset$, thus $\bigcup \{H_{\kappa \setminus B_{n}}:n\in \mathbb{N}\}=G$. Set $A_{n}=\kappa \setminus B_{n}$. Since $A_{n}\not \in \xi$ for every $n\in \mathbb{N}$, it follows that $\mbox{Int}(H_{A_{n}})=\emptyset$. Above all, we know that $\{H_{A_{n}}:n\in \mathbb{N}\}$ is a family of nowhere dense sets whose union covers $G$. The proof is completed.
\end{proof}

\begin{corollary}
If $(G, \tau, \oplus)$ is a non-discrete irresolvable left topological gyrogroup (or right topological gyrogroup) of non-measurable cardinality, then $(G,\tau,\oplus)$ is of first-category.
\end{corollary}

\begin{proof}
Let $\Delta (G,\tau,\oplus)=\kappa$. If $\kappa =\omega$, it is obvious.

Suppose that $\kappa >\omega$. There exists an open neighborhood $U$ of $0$ such that $|U|=\kappa$. Then the gyrogroup $G_{0}=\langle U\rangle$ is open in $G$ and the dispersion character of $G_{0}$ coincides with its power. However, it follows from Theorem \ref{3dl1} that $G_{0}$ is of first category. Therefore, $G$ is of first category.
\end{proof}

\begin{corollary}
Every non-discrete irresolvable topological gyrogroup of non-measurable cardinality is of first category.
\end{corollary}

A {\it (strongly) $\sigma$-discrete} space is one which is a countable union of (closed) discrete subspaces.

\begin{corollary}
Every submaximal topological gyrogroup of non-measurable cardinality is strongly $\sigma$-discrete.
\end{corollary}

\begin{proof}
It follows directly from the facts that every nowhere dense subset is closed and discrete in a submaximal space and every submaximal space is irresolvable.
\end{proof}

\smallskip
\section{Submaximal properties of strongly topological gyrogroups}
In this section, we prove that the cellularity of every submaximal strongly topological gyrogroup $G$ is equal to the cardinality of $G$. Further, we prove that every submaximal strongly topological gyrogroup of non-measurable cardinality is hereditarily paracompact.

A topological gyrogroup $(G,\tau,\oplus)$ is {\it left $\kappa$-bounded} for some cardinal $\kappa$ if for every open neighborhood $U$ of the element $0$ there exists a subset $A\subset G$ with $|A|\leq \kappa$ such that $A\oplus U=G$. First, we need some lemmas in order to obtain one of main results in this section.

\begin{lemma}\label{yl3}\cite{LF5}
Suppose that $(G,\tau ,\oplus)$ is a strongly topological gyrogroup with a symmetric neighborhood base $\mathscr U$ at $0$. Suppose further that $U,V,W$ are all open neighborhoods of $0$ such that $V\oplus V\subset W$, $W\oplus W\subset U$ and $V,W\in \mathscr U$. If a subset $A$ of $G$ is $U$-disjoint, then the family of open sets $\{a\oplus V:a\in A\}$ is discrete in $G$.
\end{lemma}

\begin{lemma}\label{4dl1}
Let $(G,\tau ,\oplus)$ be a strongly topological gyrogroup with a symmetric open neighborhood base $\mathscr U$ at $0$. If $c(G)\leq \kappa$, then $G$ is left $\kappa$-bounded.
\end{lemma}

\begin{proof}
For an arbitrary open neighborhood $U$ of the identity element $0$ in $G$, there exist $V,W\in \mathscr U$ such that $V\oplus V\subset W$ and $W\oplus W\subset U$. Let $$\mathscr F=\{A\subset G:(b\oplus V)\cap (a\oplus V)=\emptyset,\mbox{ for any distinct}\ a, b\in A\}.$$ Define $\leq $ in $G$ such that $A_{1}\leq A_{2}$ if and only if $A_{1}\subset A_{2}$, for any $A_{1},A_{2}\in \mathscr F$. Then, $(\mathscr F,\leq)$ is a poset and the union of any chain of $V$-disjoint sets is again a $V$-disjoint set. Therefore, it follows from Zorn's Lemma that there exists a maximal element $A$ in $\mathscr F$ so that $\{a\oplus V:a\in A\}$ is a disjoint family of non-empty spen sets in $G$. By Lemma \ref{yl3}, the family of open sets $\{a\oplus V:a\in A\}$ is discrete in $G$.  Since $c(G)\leq \kappa$, it follows that $|A|\leq \kappa$.

Since $A$ is maximal, for every $x\in G$, there exists $a\in A$ such that $(x\oplus V)\cap (a\oplus V)\not =\emptyset$. Then, there exist $v_{1},v_{2}\in V$ such that $x\oplus v_{1}=a\oplus v_{2}$. By the right cancellation law (2) in Lemma~\ref{a}, we have that
\begin{eqnarray}
x&=&(x\oplus v_{1})\oplus gyr[x,v_{1}](\ominus v_{1})\nonumber\\
&=&(a\oplus v_{2})\oplus gyr[x,v_{1}](\ominus v_{1})\nonumber\\
&\in &(a\oplus v_{2})\oplus gyr[x,v_{1}](V)\nonumber\\
&=&(a\oplus v_{2})\oplus V\nonumber\\
&=&a\oplus (v_{2}\oplus gyr[v_{2},a](V))\nonumber\\
&=&a\oplus (v_{2}\oplus V)\nonumber\\
&\subset &a\oplus (V\oplus V)\nonumber\\
&\subset &a\oplus U.\nonumber
\end{eqnarray}
Therefore, $A\oplus U=G$.
\end{proof}

\begin{corollary}\label{ktl1}
Every separable strongly topological gyrogroup $G$ is left $\omega$-narrow.
\end{corollary}

From \cite[Corollary 5.10]{LF4}, it follows that every pseudocompact rectifiable space is a Souslin space. Therefore, we have the following corollary.

\begin{corollary}
Every pseudocompact strongly topological gyrogroup is left $\omega$-narrow.
\end{corollary}

\begin{proposition}\label{kmt1}
If $(G,\tau ,\oplus)$ is a left $\kappa$-bounded strongly topological gyrogroup with a symmetric open neighborhood base $\mathscr U$ at $0$ and $H$ is a subgyrogroup of $G$, then $H$ is also left $\kappa$-bounded.
\end{proposition}

\begin{proof}
Let $W$ be an arbitrary open neighborhood of $0$ in $H$. Then we can fix $V\in \mathscr U$ such that $(V\oplus V)\cap H\subset W$. Since $G$ is left $\kappa$-bounded, there is a set $B$ with $|B|\leq \kappa$ in $G$ such that $B\oplus V=G$. Let $C=\{c\in B:(c\oplus V)\cap H\not =\emptyset\}$. It is obvious that $|C|\leq |B|\leq \kappa$ and $H\subset C\oplus V$. We can find $a_{c}\in (c\oplus V)\cap H$ for every $c\in C$. Let $A=\{a_{c}:c\in C\}$ and $|A|\leq \kappa$ in $H$. We show that $A\oplus W=H$.

Since $H$ is a subgyrogroup and $(V\oplus V)\cap H\subset W\subset H$, we have that $(A\oplus (V\oplus V))\cap H\subset A\oplus W$. Hence, it suffices to prove $H\subset A\oplus (V\oplus V)$. For every $c\in C$, there exists $v\in V$ such that $a_{c}=c\oplus v$. Therefore
\begin{eqnarray}
c&=&(c\oplus v)\oplus gyr[c,v](\ominus v)\nonumber\\
&=&a_{c}\oplus gyr[c,v](\ominus v)\nonumber\\
&\in &a_{c}\oplus gyr[c,v](V)\nonumber\\
&=&a_{c}\oplus V.\nonumber
\end{eqnarray}
Thus, $C\subset A\oplus V$. Moreover, since $H\subset C\oplus V$, we have that
\begin{eqnarray}
H&\subset &(A\oplus V)\oplus V\nonumber\\
&\subset&A\oplus (V\oplus gyr[V,A](V))\nonumber\\
&=&A\oplus (V\oplus V).\nonumber
\end{eqnarray}
Therefore, $H=A\oplus W$.
\end{proof}

\begin{lemma}\label{yl4}
Let $(G, \tau, \oplus)$ be a strongly topological gyrogroup with a symmetric open neighborhood base $\mathscr U$ at $0$ and $H$ a closed and discrete subgyrogroup of $G$. Take any open neighborhood $V$ of the identity element $0$ in $G$ such that $V\cap H=\{0\}$.

\smallskip
(1) If $U\in \mathscr U$ such that $U\oplus U\subset V$, then $H\not \subset A\oplus U$ for any $A\subset G$ with $|A|<|H|$.

\smallskip
(2) If $W,U\in \mathscr U$ such that $W\oplus W\subset V$ and $U\oplus U\subset W$, then the family $\{g\oplus U:g\in H\}$ is discrete in $G$.
\end{lemma}

\begin{proof}
Indeed, if there is $A\subset G$ with $H\subset A\oplus U$ such that $|A|<|H|$, then for some $a\in A$ the set $a\oplus U$ must contain at least two elements of $H$. Take any $h, g\in H\cap (a\oplus U)$ such that $h\not =g$. Then there are $v,u\in U$ such that $h=a\oplus u$ and $g=a\oplus v$. Since
\begin{eqnarray}
a&=&(a\oplus v)\oplus gyr[a,v](\ominus v)\nonumber\\
&=&g\oplus gyr[a,v](\ominus v)\nonumber\\
&\in &g\oplus gyr[a,v](U)\nonumber\\
&=&g\oplus U,\nonumber
\end{eqnarray}
it follows that
\begin{eqnarray}
h&=&a\oplus u\nonumber\\
&\in &(g\oplus U)\oplus u\nonumber\\
&\subset &(g\oplus U)\oplus U\nonumber\\
&=&g\oplus (U\oplus gyr[U,g](U))\nonumber\\
&=&g\oplus (U\oplus U)\nonumber\\
&\subset &g\oplus V.\nonumber
\end{eqnarray}
Therefore, $0\not =\ominus g\oplus h\in V\cap H$, which is a contradiction.

In order to prove (2), we take an arbitrary $g\in G$. We show that $g\oplus U$ intersects at most one element of $\{x\oplus U: x\in H\}$. Suppose not, then we can find $a, b, c, d\in U$ and distinct $p, q\in H$ such that $p\oplus a=g\oplus b$ and $q\oplus c=g\oplus d$. Since
\begin{eqnarray}
p&=&(p\oplus a)\oplus gyr[p,a](\ominus a)\nonumber\\
&=&(g\oplus b)\oplus gyr[p,a](\ominus a)\nonumber\\
&\in &(g\oplus b)\oplus gyr[p,a](U)\nonumber\\
&=&(g\oplus b)\oplus U\nonumber\\
&\subset &(g\oplus U)\oplus U\nonumber\\
&=&g\oplus (U\oplus gyr[U,g](U))\nonumber\\
&=&g\oplus (U\oplus U)\nonumber\\
&\subset &g\oplus W,\nonumber
\end{eqnarray}
and by the same method, we know that $g\in q\oplus W$. Therefore, $$p\in g\oplus W\subset (q\oplus W)\oplus W=q\oplus (W\oplus gyr[W,q](W))=q\oplus (W\oplus W)\subset q\oplus V,$$ which implies that $(\ominus q)\oplus p\in V$. Since $p$ and $q$ are different, we have $(\ominus q)\oplus p\neq 0$ and $(\ominus q)\oplus p\in V\cap H$, which is a contradiction.
\end{proof}

\begin{lemma}\label{lll}
If $G$ is a strongly topological gyrogroup, then there exists an open $L$-subgyrogroup $H$ of $G$ such that $|H|=\triangle(G)$.
\end{lemma}

\begin{proof}
Let $G$ be a strongly topological gyrogroup with a symmetric open neighborhood base $\mathscr U$ at $0$. Since $\mathscr U$ is a base at $0$ and $G$ is homogeneous, it follows from Definition~\ref{d00} that there exists an open neighborhood $V$ of $0$ such that $|V|=\triangle(G)$, then we can find $U\in\mathscr U$ such that $U\subset V$. Clearly, $|U|\leq |V|$ because $U\subset V$, hence $|U|=\triangle(G)$. Put $H_{0}=U$. We define a sequence $\{H_{n}\}_{n\in\mathbb{N}}$ of subsets of $G$ such that $H_{n+1}=\ominus(H_{n}\oplus H_{n})\cup (H_{n}\oplus H_{n})$ for each $n\in\mathbb{N}$. Set $H=\bigcup_{n\in\mathbb{N}}H_{n}$, hence $H$ is open in $G$ and $|H|=\triangle(G)$. We claim that $H$ is an $L$-subgyrogroup in $G$. It is clear that $H$ is closed for the gyrogroup operation and the inverse, so $H$ is a subgyrogroup of $G$. Moreover, for any $x, y\in G$, $\mbox{gyr}[x, y]$ is a groupoid homomorphism from $G$ onto itself and $\mbox{gyr}[x, y](U)=U$. Next we claim that $\mbox{gyr}[x, y](H_{n})=H_{n}$ for any $x, y\in G$ and $n\in\mathbb{N}$. Clearly, we have $\mbox{gyr}[x, y](H_{0})=H_{0}$ for all $x, y\in G$.
By induction, we may assume that for some $n\in\mathbb{N}$ we have $\mbox{gyr}[x, y](H_{n})=H_{n}$ for any $x, y\in G$. Now we prove $\mbox{gyr}[x, y](H_{n+1})=H_{n+1}$ for any $x, y\in G$.

Indeed, for all $x, y\in G$, we have
\begin{eqnarray}
\mbox{gyr}[x, y](H_{n+1})&=&\mbox{gyr}[x, y](\ominus(H_{n}\oplus H_{n})\cup (H_{n}\oplus H_{n}))\nonumber\\
&=&\ominus (\mbox{gyr}[x, y](H_{n})\oplus \mbox{gyr}[x,y]( H_{n}))\cup (\mbox{gyr}[x,y](H_{n})\oplus \mbox{gyr}[x,y](H_{n}))\nonumber\\
&=&\ominus (H_{n}\oplus H_{n})\cup (H_{n}\oplus H_{n})\nonumber\\
&=&H_{n+1}.\nonumber
\end{eqnarray}
 Then, for any $z\in H$, there exists $n\in \mathbb{N}$ such that $z\in H_{n}$. It follows that $\mbox{gyr}[x, y](z)\in \mbox{gyr}[x, y](H_{n})=H_{n}\subset H$. Hence $\mbox{gyr}[x, y](H)=H$. Therefore, $H$ is an $L$-subgyrogroup of $G$.
\end{proof}

Let the gyrogroup $K_{16}$ be endowed with discrete topology \cite[p. 41]{UA2002} and let $\mathbb{D}$ be topological gyrogroup in Example~\ref{lz1}. Put $G=K_{16}\times \mathbb{D}$, where $G$ is endowed with the product topology and the operation with coordinate. Fix an arbitrary $L$-subgyrogroup $H$ in $K_{16}$, for example, $H=\{0, 1, 2, 3\}$ or $\{0, 1, 2, 3, \cdots, 7\}$. Then $H$ is an open $L$-subgyrogroup of $G$ such that $|H|=\triangle(G)$.

\begin{theorem}\label{4dl2}
Let $\kappa$ be an infinite cardinal number. If $G$ is a left $\kappa$-bounded submaximal strongly topological gyrogroup, then $|G|\leq \kappa$.
\end{theorem}

\begin{proof}
Let $\triangle(G)=\lambda$, and from Lemma~\ref{lll} take an open $L$-subgyrogroup $N$ of $G$ such that $|N|=\lambda$. Then $G$ is a discrete union of left translations of $N$ and it is impossible to cover $G$ by less than $|G/N|$ left translations of $N$. Thus, $|G/N|\leq \kappa$. Then it suffices to prove $\lambda =|N|\leq \kappa$. We divide the proof into the following two cases.

\smallskip
{\bf Case 1} $\lambda$ is a limit cardinal.

\smallskip
We assume that $\kappa <\lambda$ and take a subset $P\subset N$ such that $\kappa <|P|=\gamma <\lambda$. Then $|\langle P\rangle|=\gamma$ and $H=\langle P\rangle$ is closed and discrete in $G$ since $G$ is submaximal. From (1) of Lemma \ref{yl4}, it follows that $G$ is not left $\kappa$-bounded, which is a contradiction.

\smallskip
{\bf Case 2} $\lambda$ is a successor cardinal.

\smallskip
Then $\lambda =cf(\lambda)$. Since the $L$-subgyrogroup $N$ is submaximal, it follows that $\triangle(N)=|N|$, then we can take a canonical decomposition $\{H_{\alpha } :\alpha <\lambda\}$ for $N$. By Lemma \ref{yl2}, there exists a cofinal set $A\subset \lambda$ such that $H_{A}=\bigcup \{H_{\alpha}:\alpha \in A\}$ is closed and discrete in $G$ and $0\not \in H_{A}$.

For every $\alpha \in A$, choose a point $x_{\alpha}\in H_{\alpha}$. It is obvious that $Y=\{x_{\alpha}:\alpha \in A\}$ is closed and discrete in $G$. By the confinality of $A$ in $\lambda$, we have $|Y|=\lambda$. Hence we can find an open neighborhood $U$ of $0$ such that $U\cap H_{A}=\emptyset$. There exists $V\in \mathscr U$ such that $V\oplus V\subset U$. We claim that $P\oplus V\not=G$ for any $P\subset G$ with $|P|<\lambda.$

Suppose not, then there exists $P\subset G$ with $|P|<\lambda$ such that $P\oplus V=G$. Since $|Y|=\lambda$ and $|P|<\lambda$, there exists a $p\in P$ such that $p\oplus V$ contains at least two distinct points $x_{\alpha},x_{\beta}$ of $Y$, where $\alpha <\beta$. It follows from $\ominus x_{\alpha}\in H_{\alpha}$ and $x_{\beta}\in H_{\beta}$ that $(\ominus x_{\alpha})\oplus x_{\beta}\in H_{\beta}\subset H_{A}$ by (6) of Lemma~\ref{yl1}. Moreover, we can find $u,v\in V$ such that $p\oplus u=x_{\alpha}$ and $p\oplus v=x_{\beta}$. Since $$p=(p\oplus u)\oplus gyr[p,u](\ominus u)=x_{\alpha}\oplus gyr[p,u](\ominus u)\in x_{\alpha}\oplus gyr[p,u](V)=x_{\alpha}\oplus V,$$ then $$x_{\beta}\in (x_{\alpha}\oplus V)\oplus V=x_{\alpha}\oplus (V\oplus gyr[V,x_{\alpha}](V))=x_{\alpha}\oplus (V\oplus V).$$ Thus, $\ominus x_{\alpha}\oplus x_{\beta}\in (V\oplus V)\cap H_{A}\subset U\cap H_{A}$, which is a contradiction. Thus, $P\oplus V\not =G$.
Therefore, $\kappa \geq \lambda$.
\end{proof}

Now we can easily obtain the first main result in this section.

\begin{theorem}
$c(G)=|G|$ for every submaximal strongly topological gyrogroup $G$. In particular, a submaximal strongly topological gyrogroup with the Suslin property is countable.
\end{theorem}

\begin{proof}
By Lemma~\ref{4dl1}, if $c(G)\leq \kappa$, we have that $G$ is left $\kappa$-bounded. Then it follows from Theorem \ref{4dl2} that $|G|\leq \kappa$.
\end{proof}

Finally, we prove the second main result in this section.

\begin{lemma}\cite[Lemma 3.13]{APT}\label{yl5}
Let $X$ be a regular space. Suppose that $X=\bigcup \{H_{n}:n\in \mathbb{N}\}$, where the subsets $H_{n}$ have the following properties:

\smallskip
(1) $H_{i}$ is closed and discrete in $X$ for all $i\in \mathbb{N}$;

\smallskip
(2) $H_{i}\cap H_{j}=\emptyset$ if $i\not =j$;

\smallskip
(3) for every $x\in X$ there is an open neighborhood $V_{x}$ of $x$ such that for any $i\in \mathbb{N}$ the family $\{V_{x}:x\in H_{i}\}$ is discrete in $X$.

\smallskip
Then $X$ is weakly collectionwise Hausdorff.
\end{lemma}

\begin{theorem}
Let $(G,\tau ,\oplus)$ be a submaximal strongly topological gyrogroup with a symmetric open neighborhood base $\mathscr U$ at $0$. If $G$ has non-measurable cardinality, then $G$ is hereditarily paracompact.
\end{theorem}

\begin{proof}
From Lemma~\ref{lll}, $G$ has an open $L$-subgyrogroup whose cardinality and dispersion character are
equal. Hence, if we prove hereditary paracompactness of this open $L$-subgyrogroup of $G$, then $G$ will be
hereditarily paracompact. Without loss of generality, we may assume that $|G|=\kappa$ is a uncountable cardinal such that $\triangle(G)=\kappa$. From \cite[Theorem 2.2]{APT}, each submaximal weakly collectionwise Hausdorff space is hereditarily paracompact, hence it suffices to prove that $G$ is weakly collectionwise Hausdorff. Take a canonical decomposition $\{H_{\alpha}:\alpha <\kappa\}$ of $G$. We verify that a family of subsets $\{P_{n}: n\in \mathbb{N}\}$ satisfies the conditions of Lemma~\ref{yl5}.

Indeed, it follows from Theorem \ref{3dl1} that there exists a family $\{A_{n}:n\in \mathbb{N}\}$ of subsets of $\kappa$ such that $\bigcup \{A_{n}:n\in \mathbb{N}\}=\kappa$ and $P_{n}=H_{A_{n}}=\bigcup \{H_{\alpha}:\alpha \in A_{n}\}$ is closed and discrete in $G$. Now we only need to check the condition (3) in Lemma~\ref{yl5}.

For each $n\in \mathbb{N}$ there exists an open neighborhood $U_{n}$ of $0$ with $U_{n}\cap H_{A_{n}}\subset \{0\}$. Choose $V_{n},W_{n}\in \mathscr U$ such that $V_{n}\oplus V_{n}\subset W_{n}$ and $W_{n}\oplus W_{n}\subset U_{n}$. For every $\alpha \in A_{n}$, let $W_{\alpha}^{n}=H_{\alpha}\oplus V_{n}$. For arbitrary $g\in G$, we show that $O=g\oplus V_{n}$ can intersect at most one element of $\gamma _{n}=\{W_{\alpha}^{n}:\alpha \in A_{n}\}$. We assume that there are $\alpha ,\beta \in A_{n}$ with $\alpha <\beta$ such that $W_{\alpha}^{n}\cap O\not=\emptyset \not =W_{\beta}^{n}\cap O$. Therefore, there are $p\in H_{\alpha},~q\in H_{\beta}$ and $u,v,u_{1},v_{1}\in V_{n}$ such that $g\oplus v_{1}=p\oplus v$ and $g\oplus u_{1}=q\oplus u$. Since
\begin{eqnarray}
g&=&(g\oplus u_{1})\oplus gyr[g,u_{1}](\ominus u_{1})\nonumber\\
&=&(q\oplus u)\oplus gyr[g,u_{1}](\ominus u_{1})\nonumber\\
&\in &(q\oplus u)\oplus gyr[g,u_{1}](V_{n})\nonumber\\
&\subset &(q\oplus V_{n})\oplus V_{n}\nonumber\\
&=&q\oplus (V_{n}\oplus gyr[V_{n},q](V_{n})\nonumber\\
&=&q\oplus (V_{n}\oplus V_{n})\nonumber\\
&\subset &q\oplus W_{n}.\nonumber
\end{eqnarray}
Then, by the same method, we have
\begin{eqnarray}
p&\in &g\oplus W_{n}\nonumber\\
&\subset &(q\oplus W_{n})\oplus W_{n}\nonumber\\
&=&q\oplus (W_{n}\oplus gyr[W_{n},q](W_{n}))\nonumber\\
&=&q\oplus (W_{n}\oplus W_{n})\nonumber\\
&\subset &q\oplus U_{n}.\nonumber
\end{eqnarray}
Therefore, $0\not =(\ominus q)\oplus p\in U_{n}$. Moreover, $\ominus q\in H_{\beta}$ implies $(\ominus q)\oplus p\in H_{\beta}$ by (6) of Lemma~\ref{yl1}. It is contradict with $H_{A_{n}}\cap U_{n}\subset \{0\}$. Thus, $\gamma _{n}=\{W_{\alpha}^{n}:\alpha \in A_{n}\}$ is discrete in $G$.

For each $\alpha \in A_{n}$, since $|G_{\alpha}|<\kappa$, we have that the subgyrogroup $G_{\alpha}=\bigcup \{H_{\upsilon}:\upsilon \leq \alpha\}$ is closed and discrete in $G$. Let $U$ be an open neighborhood of $0$ such that $U\cap G_{\alpha}=\{0\}$. We can find $V,W\in \mathscr U$ such that $V\oplus V\subset W$ and $W\oplus W\subset U$. It follows from (2) of Lemma \ref{yl4} that the family $\mu =\{x\oplus V:x\in G_{\alpha}\}$ is discrete.

For an arbitrary $n\in \mathbb{N}$, if $\alpha \in A_{n}$ and $p\in H_{\alpha}$, let $V_{p}=(p\oplus V)\cap W_{\alpha}^{n}$. It is clear that $\mu _{n}=\{V_{p}:p\in P_{n}\}$ is discrete in $G$ and $p\in V_{p}$ for every $p\in P_{n}$.

Therefore, the conditions (1)-(3) in Lemma \ref{yl5} hold and it follows from Lemma \ref{yl5} that $G$ is weakly collectionwise Hausdorff.
\end{proof}

By \cite[Theorem 7.2]{E}, we know that if a normal space is a countable union of its strongly zero-dimensional spaces, then it is strongly zero-dimensional. Therefore, we have the following results.

\begin{corollary}
If a submaximal strongly topological gyrogroup $G$ has non-measurable cardinality, then $dim(G)=0$. In particular, $G$ cannot be connected.
\end{corollary}

\begin{corollary}
If there does not exist any measurable cardinal, then every submaximal strongly topological gyrogroup is hereditarily paracompact and zero-dimensional in the sense of the dimension dim. In particular, no submaximal infinite strongly topological gyrogroup is connected.
\end{corollary}

Since every strongly topological gyrogroup is a topological gyrogroup, it is natural to pose the following question.

\begin{question}
Is each submaximal topological gyrogroup $G$ of non-measurable cardinality hereditarily paracompact?
\end{question}

\smallskip
{\bf Acknowledgements}. We wish to thank anonymous referees for the detailed list of corrections, suggestions to the paper, and all her/his efforts
in order to improve the paper.

\end{document}